\documentclass[a4paper,11pt,leqno]{amsart}
\usepackage{amsmath}
\usepackage{amssymb}
\usepackage{amsthm}
\usepackage{indentfirst}
\usepackage[english]{babel}
\usepackage[T1]{fontenc}  
\usepackage{breqn}
\usepackage{cite}    
\usepackage{mathrsfs}

\theoremstyle{plain} 
\newtheorem{thm}{Theorem}[section] 
\newtheorem{cor}[thm]{Corollary} 
\newtheorem{lem}[thm]{Lemma}

\theoremstyle{definition}

\newcommand{\wt}{\widetilde}
\renewcommand{\phi}{\varphi}

\numberwithin{equation}{section}

\email{danymastro93@hotmail.it}

\begin{document}

\begin{center}
\textbf{{\Large Weighted Average Number of Prime $m$-tuples}}
\end{center}
\begin{center}
\textbf{{\Large lying on an Admissible $k$-tuple of Linear Forms}}
\end{center} \ \
\begin{center}
{\large Daniele Mastrostefano}
\end{center}\ 
\\
\begin{center}
\textbf{Abstract}
\end{center}
In this paper we find an upper bound for the sum $\sum_{x<n\leq 2x}\textbf{1}_{\mathbb{P}}(n+h_{i_{1}})\cdots\textbf{1}_{\mathbb{P}}(n+h_{i_{m+1}})w_{n}$, where $(h_{i_{1}},...,h_{i_{m+1}})$ is any $(m+1)$-tuple of elements in the admissible set $\mathcal{H}=\{h_{1},...,h_{k}\}$, $m\geq 1$ and $x$ is sufficiently large, with the same weights $w_{n}$ used in the Maynard's paper ``Dense clusters of primes in subsets''. The estimate will be uniform over positive integer $k$ with $m+1\leq k\leq (\log x)^{1/5}$ and on admissible set $\mathcal{H}$ with $0\leq h_{1}<...<h_{k}\leq x$. Moreover, we make explicit the dependence on $m$. The upper bound will depend on an integral of a smooth function and on the singular series of $\mathcal{H}$, which naturally arises in this context.\\
\\
\\
\section{Introduction}
Let us fix $\mathbb{P}$ the set of prime numbers, $k$ a positive integer and $\mathcal{H}=\{h_{1},...,h_{k}\}$ an admissible set. In the paper ``Small gaps between primes''\cite{M1}, Maynard proved that there exist infinitely many bounded intervals containing at least $m\geq 1$ primes, showing that the weighted sum 
\begin{equation}
\label{eq: 1.1}
S=\sum_{x<n\leq 2x}\left(\sum_{i=1}^{k}\textbf{1}_{\mathbb{P}}(n+h_{i})-(m-1)\right)w_{n}
\end{equation}
is positive, when $x$ is sufficiently large. Here $\textbf{1}_{\mathbb{P}}(n)$ is the characteristic function of the set of prime numbers and $w_{n}$ are chosen as non-negative smooth $k$-dimensional Selberg sieve weights. In fact, if $S>0$, there must exist an integer $n\in [x,2x]$ such that the corresponding factor in parentheses is positive, which is equivalent to say that at least $m$ prime numbers lye on the translates $n+h_{1},...,n+h_{k}$. Since $\mathcal{H}$ is fixed and we can vary $x$, we obtain the aforementioned result.

In the subsequent paper ``Dense clusters of primes in subsets''\cite{M2}, Maynard proved that a uniform version of the sieve method \eqref{eq: 1.1} can lead to improve the above result, finding a lower bound on the number of integers $n\in [x,2x]$ for which there are at least $m$ primes among $n+h_{1},...,n+h_{k}$. More specifically, he showed that 
$$\#\{n\in [x,2x]: \#(\{n+h_{1},...,n+h_{k}\}\cap \mathbb{P})\geq m\}\gg_{k} \frac{x}{\log^{k}x}.$$
This is \cite[Theorem 3.1]{M2}. The estimate holds with some uniformity on the parameters $m$ and $k$ and on the admissible set $\mathcal{H}$, which vary with $x$ in certain ranges.

In order to obtain such result, Maynard estimated various sums involved in a generalization on the sieve method (1.1), which is essentially of the following form
\begin{equation}
\label{eq: 1.2}
S'=\sum_{x<n\leq 2x}\left(\sum_{i=1}^{k}\textbf{1}_{\mathbb{P}}(n+h_{i})-(m-1)-k\textbf{1}_{\mathcal{B}}(n)\right)w_{n},
\end{equation}
where $\textbf{1}_{\mathcal{B}}$ is the characteristic function of the set of integers $\mathcal{B}$. Maynard gave estimates on these new particular sums in \cite[Proposition 6.1]{M2}. To find a lower bound for $S'$, useful for application, he needed to obtain a lower bound on the weighted average number of primes lying on the admissible set $\mathcal{H}$. Indeed, he showed that 
\begin{equation}
\label{eq: 1.3}
\sum_{x<n\leq 2x} \textbf{1}_{\mathbb{P}}(n+h)w_{n}(\mathcal{H})\gg\left(\frac{B}{\phi(B)•}\right)^{k-1}\mathfrak{S}_{B}(\mathcal{H})x(\log x)^{k}\frac{\log k}{k}I_{k}(F),
\end{equation}
for every $h\in\mathcal{H}$ and where $B$ is a suitable positive integer. Here $\mathfrak{S}_{B}(\mathcal{H})$ is the singular series associated with the set $\mathcal{H}$ and 
$$I_{k}(F)=\int_{0}^{\infty}\cdots\int_{0}^{\infty} F^{2}(t_{1},...,t_{k})dt_{1}\cdots dt_{k},$$
with $F$ a smooth function $F:\mathbb{R}^{k}\rightarrow \mathbb{R}$ depending only on $k$.

The aim of our paper is to find an upper bound for the generalization of the sum \eqref{eq: 1.3} to $m$-tuples of primes. For any admissible set of linear functions $\mathcal{L}=\{L_{1}(n),...,L_{k}(n)\}$ $=\{n+h_{1},...,n+h_{k}\}$ and for every prime numbers $p$ and integer $B$, let us define the function $\omega(p)$ as
$$\omega(p)= \left\{ \begin{array}{ll}
        \#\{1\leq n\leq p: \prod_{i=1}^k L_{i}(n)\equiv 0\pmod{p}\} & \mbox{if $p\nmid B$};\\
        0 & \mbox{$p|B$}.\end{array} \right. $$
Moreover, we define for every integer $D$ the singular series attached to $\mathcal{L}$ as
$$\mathfrak{S}_D(\mathcal{L})=\prod_{p\nmid D}\Bigl(1-\frac{\omega(p)}{p}\Bigr)\Bigl(1-\frac{1}{p}\Bigr)^{-k}.$$
From the admissibility of $\mathcal{L}$ and the definition of $\omega(p)$ it follows that $\mathfrak{S}_D(\mathcal{L})$ converges.

Finally, we put $W=\prod_{p\leq 2k^{2}, p\nmid B}p$ and we let $W_{1},...,W_{k}$ to be square-free integers each a multiple of $WB$, such that any prime $p\nmid WB$ divides exactly $k-\omega(p)$ of them. Now, we can state the main result.
\begin{thm}
\label{thm: 1.1} 
For any $m\geq 1$ and $0\leq h_{1}<...<h_{k}\leq x$, and if $k$ is a sufficiently large positive integer with $m+1\leq k\leq (\log x)^{1/5}$, we have 
\begin{equation}
\label{eq: 1.4}
\sum_{x<n\leq 2x} \mathbf{1}_{\mathbb{P}}(L_{i_{1}}(n))\cdots \mathbf{1}_{\mathbb{P}}(L_{i_{m+1}}(n))w_{n}\leq C^{m}\bar{y}^{2}\frac{\varphi(W)^{k-1}}{W^{k-1}}\frac{(WB)^{k-2m-1}}{\varphi(WB)^{k-2m-1}}\mathfrak{S}_{B}(\mathcal{L})\frac{(\log k)^{m+1}}{k^{m+1}}I_{k}(F)x(\log x)^{k},
\end{equation}
for any $1\leq i_{1}<...<i_{m+1}\leq k$, for a certain constant $C>0$ and for any $x$ large enough. Here, we take $B$ as a suitable integer such that $B/\phi(B)=1+o(1)$ and we let $\bar{y}$ to be 
$$\bar{y}= (3m)^{m}\left(\frac{(W_{i_{1}},...,W_{i_{m}})}{\phi((W_{i_{1}},...,W_{i_{m}}))}\right)^{m}\prod_{p:\ n(p)\geq 1}\left(1+\frac{n(p)}{p-1}\right)^{-1}\left(1-\frac{1}{p•}\right)^{-m},$$
where $n(p)=\#\{j\in\{1,...,m\}: p\nmid W_{i_{j}}\}$. Moreover, $I_{k}(F)$ is defined as above and the weights $w_{n}$ are the same used in \cite{M2}.
\end{thm}
Unfortunately, Theorem 1.1 is of difficult application due to the strong dependence on the factors $W_{i_{1}},...,W_{i_{m}}.$ However, estimating carefully the product present in $\bar{y}$ and averaging over all the $(m+1)$-tuples $(h_{i_{1}},...,h_{i_{m+1}})$ could lead to cancellations and we could obtain a bound free of these terms.
Anyway, adding some restrictions we can simplify the result above in the following corollary.
\begin{cor}
\label{cor 1.2} If we assume that the elements of the admissible set $\mathcal{L}=\{n+h_{1},...,n+h_{k}\}$ verify also that $h_{1},...,h_{k}\leq k^{2}$, we have
\begin{equation}
\label{eq: 1.5}
\sum_{x<n\leq 2x} \mathbf{1}_{\mathbb{P}}(L_{i_{1}}(n))\cdots\mathbf{1}_{\mathbb{P}}(L_{i_{m+1}}(n))w_{n}\leq D^{m}m^{m}\left(\frac{B}{\varphi(B)}\right)^{k-1}\mathfrak{S}_{B}(\mathcal{L})x(\log x)^{k}\left(\frac{\log k}{k•}\right)^{m+1}I_{k}(F),
\end{equation}
for any $1\leq i_{1}<...<i_{m+1}\leq k$, for a suitable absolute constant $D>0$ and if $x$ is sufficiently large.
\end{cor}
Comparing the estimate \eqref{eq: 1.5} with \eqref{eq: 1.3}, we see that they are of the same correct order of magnitude.
The motivation in proving the Theorem 1.1 and its Corollary relies on the analogy with \cite{BFM}. In that occasion was considered an admissible set of linear functions $\mathcal{L}=\{n+h_{1},...,n+h_{k}\},$ with $n\in [N,2N]$ and with $h_{j}-h_{i}\asymp \log N$. The elements of the $k$-tuple $\mathcal{H}=\{h_{1},...,h_{k}\}$ were allowed to grow with $N$ and were chosen weights $w_{n}$ suitable to such uniform situation. Under these circumstances, Banks, Freiberg and Maynard found an upper bound for the sum
$$\sum_{x<n\leq 2x}\textbf{1}_{\mathbb{P}}(n+h_{i})\textbf{1}_{\mathbb{P}}(n+h_{j})w_{n},$$
for each pair $i\neq j\in \{1,...,k\}$, which, inserted in a sieve method like \eqref{eq: 1.2}, led them to obtain $m$-tuples of primes where each of these primes belongs to a different subset, of a prescribed partition of $\mathcal{H}$, containing no other prime numbers.

Combining this with an Erd\H{o}s--Rankin construction, the authors of \cite{BFM} found information on the percentage of limit points of the sequence of normalized prime gaps in the set of positive real numbers. 
We are confident that the bound \eqref{eq: 1.5} can find applications in the context of the sieve method introduced in \cite{M2}, joined to the other results proved in \cite[Proposition 6.1]{M2}. Perhaps, an explicit version of \eqref{eq: 1.5} could be useful exactly regarding the study of limit points of the sequence of normalized prime gaps, in a way similar to \cite{BFM}.

The proof of Theorem \ref{thm: 1.1} is based on computations and ideas coming from the Maynard's paper \cite{M2}. Therefore, we borrow from it the notations and the main definitions, which we rewrite in section 2 and 3 for completeness, following closely the presentation in \cite{M2}.  
\section{Notations}
We consider $0<\theta<1$ a fixed real constant and $m\geq 1$ a positive integer. All asymptotic notation such as $O(\cdot), o(\cdot), \ll, \gg$ should be interpreted as referring to the limit $x\rightarrow\infty$, and any constants (implied by $O(\cdot)$) may depend on $\theta$ or $m$, but no other variable, unless otherwise noted.

Let $k=\#\mathcal{L}\geq m+1$ be the size of $\mathcal{L}=\{L_1,\dots,L_k\}$ an admissible set of integer linear functions of the form $L_i(n)=n+h_i$. Moreover, $B$ will be an integer, and $x,k$ will always to be assumed sufficiently large (in terms of $\theta$ and $m$).

All sums, products and suprema will be assumed to be taken over variables lying in the natural numbers $\mathbb{N}=\{1,2,\dots\}$ unless specified otherwise. The exception to this is when sums or products are over a variable $p$, which instead will be assumed to lie in the prime numbers $\mathbb{P}=\{2,3,\dots,\}$.

Throughout the paper, $\phi$ will denote the Euler totient function, $\tau_{r}(n)$ the number of ways of writing $n$ as a product of $r$ natural numbers and $\mu$ the Moebius function. We let $\#\mathcal{A}$ denote the number of elements of a finite set $\mathcal{A}$, and $\mathbf{1}_{\mathcal{A}}(x)$ the indicator function of $\mathcal{A}$ (so $\mathbf{1}_{\mathcal{A}}(x)=1$ if $x\in\mathcal{A}$, and 0 otherwise). We let $(a,b)$ be the greatest common divisor of integers $a$ and $b$, and $[a,b]$ the least common multiple of integers $a$ and $b$. (For real numbers $x,y$ we also use $[x,y]$ to denote the closed interval. The usage of $[\cdot,\cdot]$ should be clear from the context.)

To simplify notation we will use vectors in a way which is somewhat non-standard. In fact, $\mathbf{d}$ will denote a vector $(d_1,\dots,d_k)\in\mathbb{N}^k$. Given a vector $\mathbf{d}$, when it does not cause confusion, we write $d=\prod_{i=1}^kd_i$. Given $\mathbf{d},\mathbf{e}$, we will let $[\mathbf{d},\mathbf{e}]=\prod_{i=1}^k[d_i,e_i]$ be the product of least common multiples of the components of $\mathbf{d},\mathbf{e}$, and similarly let $(\mathbf{d},\mathbf{e})=\prod_{i=1}^k(d_i,e_i)$ be the product of greatest common divisors of the components, and $\mathbf{d}|\mathbf{e}$ denote the $k$ conditions $d_i|e_i$ for each $1\le i\le k$.

\section{Main definitions}
We recall that we are given an admissible set $\mathcal{L}=\{L_1,\dots,L_k\}=\{n+h_1,...,n+h_k\}$ of integer linear functions, an integer $B$ and quantities $R,x$. We assume that $0\leq h_{1}<...<h_{k}\leq x$ and $k$ is sufficiently large in terms of $m$ and satisfies $m+1\leq k\leq (\log x)^{1/5}.$ Moreover, we fix $R$ as $R=x^{\theta/3}$. 

We define the multiplicative functions $\omega=\omega_{\mathcal{L}}$ and $\phi_\omega=\phi_{\omega,\mathcal{L}}$ and the singular series $\mathfrak{S}_D(\mathcal{L})$ for an integer $D$ by
\begin{equation}
\label{eq:3.1}
\omega(p)= \left\{ \begin{array}{ll}
        \#\{1\leq n\leq p: \prod_{i=1}^k L_{i}(n)\equiv 0\pmod{p}\} & \mbox{if $p\nmid B$};\\
        0 & \mbox{$p|B$}.\end{array} \right. 
\end{equation}
\begin{equation}
\label{eq: 3.2}
\phi_\omega(d)=\prod_{p|d}(p-\omega(p)),
\end{equation}
\begin{equation}
\label{eq:3.3}
\mathfrak{S}_D(\mathcal{L})=\prod_{p\nmid D}\Bigl(1-\frac{\omega(p)}{p}\Bigr)\Bigl(1-\frac{1}{p}\Bigr)^{-k}.
\end{equation}
Since $\mathcal{L}$ is admissible, we have $\omega(p)<p$ for all $p$ and so $\phi_\omega(n)>0$ and $\mathfrak{S}_D(\mathcal{L})> 0$ for any integer $D$. Since $\omega(p)=k$ for all $p\nmid \prod_{i\ne j}(h_j-h_i)$ we see the product $\mathfrak{S}_D(\mathcal{L})$ converges. 
We will consider sieve weights $w_n=w_n(\mathcal{L})$, which are defined to be 0 if $\prod_{i=1}^kL_i(n)$ is a multiple of any prime $p\le 2k^2$ with $p\nmid B$. We let $W=\prod_{p\le 2k^2,p\nmid B}p$. If $(L_i(n),W)=1$ for all $1\le i\le k$ we have
\begin{equation}
\label{eq:3.4}
w_n=\Bigl(\sum_{d_i|L_i(n), \forall i}\lambda_{\mathbf{d}}\Bigr)^2,
\end{equation}
for some real variables $\lambda_{\mathbf{d}}$ depending on $\mathbf{d}=(d_1,\dots,d_k)$. We first restrict our $\lambda_{\mathbf{d}}$ to be supported on $\mathbf{d}$ with $d=\prod_{i=1}^k d_i$ square-free and coprime to $WB$.

Given a prime $p\nmid WB$, let $1\le r_{p,1}<\dots< r_{p,\omega(p)}\le p$ be the $\omega(p)$ residue classes for which $\prod_{i=1}^k L_i(n)$ vanishes modulo $p$. For each such prime $p$, we fix a choice of indices $j_{p,1},\dots,j_{p,\omega(p)}\in\{1,\dots,k\}$ such that $j_{p,i}$ is the smallest index such that
\begin{equation}
\label{eq:3.5}
L_{j_{p,i}}(r_{p,i})\equiv0\pmod{p}
\end{equation}
for each $i\in\{1,\dots,\omega(p)\}$. All the functions $L_i$ are linear and, since $\mathcal{L}$ is admissible, none of the $L_i$ are a multiple of $p$. This means that for any $L\in\mathcal{L}$ there is at most one residue class for which $L$ vanishes modulo $p$. Thus the indices $j_{p,1},\dots, j_{p,\omega(p)}$ we have chosen must be distinct. We now restrict the support of $\lambda_{\mathbf{d}}$ to $(d_j,p)=1$ for all $j\notin\{j_{p,1},\dots,j_{p,\omega(p)}\}$.

We see these restrictions are equivalent to the restriction that the support of $\lambda_{\mathbf{d}}$ must lie in the set
\begin{equation}
\label{eq:3.6}
\mathcal{D}_k=\mathcal{D}_k(\mathcal{L})=\{\mathbf{d}\in\mathbb{N}^k:\mu^2(d)=1, (d_j,W_j)=1, \forall j\},
\end{equation}
where $W_j$ are square-free integers each a multiple of $WB$, and any prime $p\nmid WB$ divides exactly $k-\omega(p)$ of the $W_j$ (such $p|W_j$ if $j\notin\{j_{p,1},\dots,j_{p,\omega(p)}\}$).

The key point of these restrictions is so that different components of different $\mathbf{d}$ occurring in our sieve weights will be relatively prime. Indeed, let $\mathbf{d}$ and $\mathbf{d'}$ both occur in the sum \eqref{eq:3.4}. If $p|d_i$ then $p|L_i(n)$, and so $i$ must be the chosen index for the residue class $n$ $\pmod{p}$. But if we also have $p|d'_j$ then similarly $j$ must be the chosen index for this residue class, and so we must have $i=j$. Hence $(d_i,d'_j)=1$ for all $i\ne j$.

We define $\lambda_{\mathbf{d}}$ in terms of variables $y_{\mathbf{r}}$ supported on $\mathbf{r}\in\mathcal{D}_k$ by
\begin{equation}
\label{eq:3.7}
\lambda_{\mathbf{d}}=\mu(d)d\sum_{\mathbf{d}|\mathbf{r}}\frac{y_{\mathbf{r}}}{\phi_\omega(r)},\qquad y_{\mathbf{r}}=\frac{\mathbf{1}_{\mathcal{D}_k}(\mathbf{r})W^kB^k}{\phi(W B)^k}\mathfrak{S}_{WB}(\mathcal{L})F\Bigl(\frac{\log{r_1}}{\log{R}},\dots,\frac{\log{r_k}}{\log{R}}\Bigr),
\end{equation}
where $F:\mathbb{R}^k\rightarrow\mathbb{R}$ is a smooth function given by
\begin{equation}
\label{eq:3.8}
F(t_1,\dots,t_k)=\psi\Bigl(\sum_{i=1}^kt_i\Bigr)\prod_{i=1}^k\frac{\psi(t_i/U_k)}{1+T_kt_i},\qquad T_k=k\log{k},\qquad U_k=k^{-1/2}.
\end{equation}
Here $\psi:[0,\infty)\rightarrow[0,1]$ is a fixed smooth non-increasing function supported on $[0,1]$ which is $1$ on $[0,9/10]$. In particular, we note that this choice of $F$ is non-negative, and that the support of $\psi$ implies that
\begin{equation}
\label{eq:3.9}
\lambda_{\mathbf{d}}=0\quad\text{if $\textstyle d=\prod_{i=1}^kd_i>R$.}
\end{equation}
We will find it useful to also consider the closely related functions $F_1$ and $F_2$ which will appear in our error estimates, defined by
\begin{equation}
\label{eq:3.10}
 F_1 (t_1,\dots,t_k)=\prod_{i=1}^k \frac{\psi(t_i/U_k)}{1+T_kt_i},\qquad 
 F_2 (t_1,\dots,t_k)=\sum_{1\le j\le k}\Bigl(\frac{\psi(t_j/2)}{1+T_kt_j}\prod_{\substack{1\le i\le k\\ i\ne j}}\frac{\psi(t_i/U_k)}{1+T_kt_i}\Bigr).
\end{equation}
Finally, by Moebius inversion, we see that \eqref{eq:3.7} implies that for $\mathbf{r}\in\mathcal{D}_k$
\begin{equation}
\label{eq:3.11}
y_{\mathbf{r}}=\mu(r)\phi_\omega(r)\sum_{\mathbf{r}|\mathbf{f}}\frac{y_{\mathbf{f}}}{\phi_\omega(f)}\sum_{\substack{\mathbf{d}\\\mathbf{r}|\mathbf{d},\mathbf{d}|\mathbf{f}}}\mu(d)=\mu(r)\phi_\omega(r)\sum_{\mathbf{r}|\mathbf{d}}\frac{\lambda_{\mathbf{d}}}{d}.
\end{equation}
\section{Proof of Theorem 1.1}
The main aim of this section is to prove the estimate \eqref{eq: 1.4}, which is the heart of Theorem \ref{thm: 1.1}. We start considering the following multidimensional Selberg bound:
\begin{equation}
\label{eq:4.1}
\textbf{1}_{\mathbb{P}}(n+h_{i_{1}})\cdots\textbf{1}_{\mathbb{P}}(n+h_{i_{m}})\leq \frac{1}{\widetilde{\lambda}^{2}_{(1,...,1)}}\bigg(\sum_{e_{1}|n+h_{i_{1}},..., e_{m}|n+h_{i_{m}}}\wt{\lambda}_{\textbf{e}}\bigg)^{2},
\end{equation}
where $\widetilde{\lambda}_{\textbf{e}}$ is a real function, with $\widetilde{\lambda}_{(1,...,1)}\neq 0$, supported on the set
\begin{equation}
\label{eq:4.2}
\mathcal{E}_{m}=\lbrace \textbf{e}\in\mathbb{N}^{m}: e<R^{\frac{1}{3}}, \mu^{2}(e)=1,\ \textrm{and}\ (e_{j}, W_{i_{j}})=1, \forall j=1,...,m\rbrace.
\end{equation}
Inserting the upper bound \eqref{eq:4.1} in the sum \eqref{eq: 1.4}, which now we consider restricted on the arithmetic progression $n\equiv v_{0}\pmod{W}$ with $v_{0}$ such that $(\prod_{i=1}^{k}L_{i}(v_{0}), W)=1$, expanding $w_{n}$ using \eqref{eq:3.4} and swapping the order of summation, we find
\begin{equation}
\label{eq:4.3}
\sum_{\substack{x< n\leq 2x\\ n\equiv v_{0}\pmod{W}}}\textbf{1}_{\mathbb{P}}(n+h_{i_{1}})\cdots\textbf{1}_{\mathbb{P}}(n+h_{i_{m+1}})w_{n}(\mathcal{L})
\end{equation}
$$\leq \frac{1}{\widetilde{\lambda}^{2}_{(1,...,1)}}\sum_{\textbf{e},\textbf{e}'\in \mathcal{E}_{m}} \wt{\lambda}_{\textbf{e}}\wt{\lambda}_{\textbf{e}'}\sum_{\substack{\textbf{d},\textbf{d}'\in \mathcal{D}_{k}\\d_{i_{j}}=d'_{i_{j}}=1,\forall j=1,...,m+1} }\lambda_{\textbf{d}}\lambda_{\textbf{d}'}\sum_{\substack{x<n\leq 2x\\ n\equiv v_{0}\pmod{W}\\n\equiv -h_{i}\pmod{[d_{i},d^{\prime}_{i}]},\forall i\\ n\equiv -h_{i_{j}}\pmod{[e_{j},e'_{j}]},\forall j}}\textbf{1}_{\mathbb{P}}(n+h_{i_{m+1}}).$$
We note that the sum over $\textbf{d}$ and $\textbf{d}'$ is restricted to have $d_{i_{j}}=d'_{i_{j}}=1,$ for every $j=1,...,m+1$, otherwise the sum on the left hand side of \eqref{eq:4.3} contributes to $0$, by the support of our weights.

We have no contribution unless $(d_{i}d'_{i},d_{j}d'_{j})=(e_{i}e'_{i},e_{j}e'_{j})=1,\forall i\neq j$ and $(d_{i}d'_{i},e_{j}e'_{j})=1, \forall i,j$. In fact, if we had $p|(d_{i}d'_{i},d_{j}d'_{j})$ or $p|(e_{i}e'_{i},e_{j}e'_{j})$ or again $p|(d_{i}d'_{i},e_{j}e'_{j})$, for suitable $i,j$, then we would find two different indices $a,b$ for which $p|n+h_{a}$ and $p|n+h_{b}$. By the support of our variables, it implies that $a$ and $b$ should be the chosen indices for the residue class $n\pmod{p}$ and therefore they would be equal. Again by the support of $\lambda_{\textbf{d}}$ and of $\wt{\lambda}_{\textbf{e}}$, we have that $(d_{i}d_{i}', W)=(e_{j}e_{j}',W)=1$, for any $i$ and $j$. We see that we can combine the congruence conditions by the Chinese remainder theorem, and the inner sum in the second line of \eqref{eq:4.3} becomes
\begin{equation}
\label{eq:4.4}
\sum_{\substack{x<n\leq 2x\\ n\equiv a\pmod{q}}}\textbf{1}_{\mathbb{P}}(n+h_{i_{m+1}})=\frac{\sum_{x< n\leq 2x}\textbf{1}_{\mathbb{P}}(n+h_{i_{m+1}})}{\varphi(W)\varphi([\textbf{d},\textbf{d}'])\varphi([\textbf{e},\textbf{e}'])}+ O(E_{q}),
\end{equation} 
for some $a$ coprime with $q=W\prod_{i}[d_{i},d_{i}']\prod_{j}[e_{j},e_{j}']$ and where 
$$E_{q}=\max_{(a,q)=1}\bigg|\sum_{\substack{x< n\leq 2x\\ n\equiv a\pmod{q}}}\textbf{1}_{\mathbb{P}}(n+h_{i_{m+1}})-\frac{\sum_{x< n\leq 2x}\textbf{1}_{\mathbb{P}}(n+h_{i_{m+1}})}{\varphi(q)}\bigg|.$$ 
Thus, the principal contribution in the estimate of \eqref{eq:4.3} comes from
\begin{equation}
\label{eq: 4.5}\frac{1}{\wt{\lambda}_{(1,...,1)}^{2}}\frac{1}{\varphi(W)}\sum_{x< n\leq 2x}\textbf{1}_{\mathbb{P}}(n+h_{i_{m+1}})\sum_{\textbf{e},\textbf{e}'\in \mathcal{E}_{m}} \frac{\wt{\lambda}_{\textbf{e}}\wt{\lambda}_{\textbf{e}'}}{\varphi([\textbf{e},\textbf{e}'])}\sideset{}{'}\sum_{\substack{\textbf{d},\textbf{d}'\in \mathcal{D}_{k}\\d_{i_{j}}=d'_{i_{j}}=1,\forall j=1,...,m+1} }\frac{\lambda_{\textbf{d}}\lambda_{\textbf{d}'}}{\varphi([\textbf{d},\textbf{d}'])}.
\end{equation}
Here we write $\Sigma^{\prime}$ for the summation with all the restrictions stated above. We note that by the Brun--Titchmarsh theorem \cite[Theorem 3.9]{MV}
\begin{equation}
\label{eq:4.6}
\sum_{x< n\leq 2x}\textbf{1}_{\mathbb{P}}(n+h_{i_{m+1}})\ll \frac{x}{\log x}.
\end{equation}
On the other hand, the error term is
\begin{equation}
\label{eq:4.7}
\ll \frac{1}{\wt{\lambda}^{2}_{(1,...,1)}}\sum_{\textbf{e},\textbf{e}'\in \mathcal{E}_{m}} |\wt{\lambda}_{\textbf{e}}\wt{\lambda}_{\textbf{e}'}|\sideset{}{'}\sum_{\substack{\textbf{d},\textbf{d}'\in \mathcal{D}_{k}\\d_{i_{j}}=d'_{i_{j}}=1,\forall j=1,...,m+1} }|\lambda_{\textbf{d}}\lambda_{\textbf{d}'}|E_{q}.
\end{equation}
To work with \eqref{eq: 4.5} and \eqref{eq:4.7}, we make some change of variables:
\begin{equation}
\label{eq:4.8}
y_{\textbf{r},\textbf{r}_{0}}=\mu(rr_{0})\varphi_{\omega}(rr_{0})\sum_{\substack{\textbf{r}|\textbf{d}\\ \textbf{r}_{0}|\textbf{e}\\(d,e)=1\\ d_{i_{j}}=1,\forall j=1,...,m+1}} \frac{\lambda_{\textbf{d}}\wt{\lambda}_{\textbf{e}}}{\varphi(d)\varphi(e)},
\end{equation}
\begin{equation}
\label{eq:4.9}
y_{\textbf{r}_{0}}=\mu(r_{0})\varphi(r_{0})\sum_{\textbf{r}_{0}|\textbf{e}}\frac{\wt{\lambda}_{\textbf{e}}}{\varphi(e)},
\end{equation}
\begin{equation}
\label{eq:4.10}
y^{(m)}_{\textbf{r}}=\mu(r)\varphi_{\omega}(r)\sum_{\substack{\textbf{r}|\textbf{d}\\d_{i_{j}}=1,\forall j=1,...,m+1}}\frac{\lambda_{\textbf{d}}}{\varphi(d)}.
\end{equation}
We note that $\textbf{r}\in\mathcal{D}_{k}$ and it is such that $r_{i_{j}}=1,\forall j=1,...,m+1$. Moreover, $\textbf{r}_{0}\in \mathcal{E}_{m}$. Finally, from the restrictions present in the sum in \eqref{eq:4.8} we may suppose that $y_{\textbf{r},\textbf{r}_{0}}=0$ if there exists a couple of components $r_{i},r_{0_{j}}$ such that $(r_{i},r_{0_{j}})>1$.

In the following lemma we concentrate on the error term \eqref{eq:4.7}. 
\begin{lem}
\label{lem 4.1}
We have 
\begin{equation}
\label{eq: 4.11}
\frac{1}{\wt{\lambda}^{2}_{(1,...,1)}}\sum_{\mathbf{e},\mathbf{e}'\in\mathcal{E}_{m}} |\wt{\lambda}_{\mathbf{e}}\wt{\lambda}_{\mathbf{e}'}|\sideset{}{'}\sum_{\substack{\mathbf{d},\mathbf{d}'\in \mathcal{D}_{k}\\d_{i_{j}}=d'_{i_{j}}=1,\forall j=1,...,m+1} }|\lambda_{\mathbf{d}}\lambda_{\mathbf{d}'}|E_{q}\ll \frac{x}{W(\log x)^{2k^2}},
\end{equation}
if $x$ is sufficiently large.
\end{lem}
\begin{proof}
By the Moebius inversion formula we can write
\begin{equation}
\label{eq: 4.12}
\wt{\lambda}_{\textbf{e}}=\varphi(e)\mu(e)\sum_{\textbf{e}|\textbf{r}_{0}\in\mathcal{E}_{m}}\frac{y_{\textbf{r}_{0}}}{\varphi(r_{0})}. 
\end{equation}
From this we easily deduce that
\begin{equation}
\label{eq: 4.13}
|\wt{\lambda}_{\textbf{e}}|\leq \varphi(e)\sum_{\textbf{e}|\textbf{r}_{0}\in\mathcal{E}_{m}}\frac{y_{\textbf{r}_{0}}\mu^{2}(r_{0})}{\varphi(r_{0})}\leq\sum_{\textbf{k}\in\mathcal{E}_{m}}\frac{y_{\textbf{k}}\mu^{2}(k)}{\varphi(k)}=\wt{\lambda}_{(1,...,1)},
\end{equation}
if we choose $y_{\textbf{r}_{0}}$ to be a positive constant, when $\textbf{r}_{0}\in \mathcal{E}_{m}$.
Using the estimate for $|\lambda_{\textbf{d}}|$ given by \cite[Lemma 8.5 (i)]{M2} jointly with \eqref{eq: 4.13}, we can estimate \eqref{eq:4.7} with
\begin{equation}
\label{eq: 4.14}
\ll (\log x)^{2k}\sum_{q\leq W R^{2}R^{2/3}, (q, B)=1} \mu^{2}(q)\tau_{3k}(q)E_{q}.
\end{equation}
Note that $W R^{2}R^{2/3}=WR^{8/3}= Wx^{\frac{8\theta}{9•}}<x^{\theta}$. We now take $\theta=1/3$. By the Landau--Page theorem (see, for example, \cite[Chapter 14]{D}) there is at most one modulus $q_0\le \exp(2c_1\sqrt{\log{x}})$, such that there exists a primitive character $\chi$ modulo $q_0$ for which $L(s,\chi)$ has a real zero larger than $1-c_2(\log{x})^{-1/2}$ (for suitable fixed constants $c_1,c_2>0$). If this exceptional modulus $q_0$ exists, we take $B$ to be the largest prime factor of $q_0$, and otherwise we take $B=1$. For all $q\le \exp(2c_1\sqrt{\log{x}})$ with $q\ne q_0$ we then have the effective bound (see, for example, \cite[Chapter 20]{D})
\begin{equation}
\label{eq: 4.15}
\phi(q)^{-1}\sideset{}{^*}\sum_{\chi}|\psi(x,\chi)|\ll x\exp(-3c_1\sqrt{\log{x}}),
\end{equation}
where the summation is over all primitive $\chi\pmod{q}$ and $\psi(x,\chi)=\sum_{n\le x}\chi(n)\Lambda(n)$.

Following a standard proof of the Bombieri--Vinogradov Theorem (see \cite[Chapter 28]{D}, for example), we have
\begin{equation}
\label{eq: 4.16}
\sum_{\substack{q<x^{1/2-\varepsilon}\\ (q,B)=1}}\sup_{(a,q)=1}\Bigl|\pi(x,q,a)-\frac{\pi(x)}{\phi(q)}\Bigr|\ll x\exp(-c_1\sqrt{\log{x}})+\log{x}\sum_{\substack{q<\exp(2c_1\sqrt{\log{x}})\\ (q,B)=1}}\sideset{}{^*}\sum_\chi\frac{|\psi'(x,\chi)|}{\phi(q)},
\end{equation}
for a certain $\varepsilon>0$, which shows that
$$\sum_{\substack{q<x^{1/3}\\ (q, B)=1}} \mu^{2}(q)E_{q}\ll \frac{x}{(\log x)^{100k^{2}}},$$
if $x$ is sufficiently large. Now, using that trivially $E_{q}\ll x/\phi(q)$ and applying Cauchy--Schwarz, we find that \eqref{eq: 4.14} is
\begin{equation}
\label{eq: 4.17}
\ll (\log x)^{2k}\bigg(\sum_{q<x^{1/3}, (q, B)=1} \mu^{2}(q)\tau_{3k}^{2}(q)E_{q}\bigg)^{1/2}\bigg(\sum_{q<x^{1/3}, (q, B)=1} \mu^{2}(q)E_{q}\bigg)^{1/2}
\end{equation}
$$\ll (\log x)^{2k}\sqrt{x}\bigg(\sum_{q<x^{1/3}, (q, B)=1} \frac{\mu^{2}(q)\tau_{3k}^{2}(q)}{\varphi(q)}\bigg)^{1/2}\bigg(\frac{x}{\log^{100k^{2}}x•}\bigg)^{1/2}.$$
This concludes the proof of the Lemma 4.1, since
$$\sum_{q<x^{1/3}, (q, B)=1} \frac{\mu^{2}(q)\tau_{3k}^{2}(q)}{\phi(q)}\leq \prod_{p\leq x^{1/3}}\left(1+\frac{9k^{2}}{p-1}\right)\leq e^{9k^{2}C}\log^{9k^{2}}x\ll (\log x)^{10k^{2}},$$
for a suitable constant $C>0$, by Mertens's Theorem \cite[Theorem 2.7]{MV}, if $x$ is sufficiently large.

We note that, if $q_0$ exists it must be square-free apart from a possible factor of at most 4, and must satisfy $q_0\gg (\log{x})/(\log\log{x})^2$. In this case we find $\log\log{x}\ll B\ll \exp(c_1\sqrt{\log{x}})$. Thus, whether or not $q_0$ exists, we have $B/\phi(B)=1+o(1)$.
\end{proof}
After the change of variables we are left with the estimate of \eqref{eq: 4.5}. The double sum is equal to
\begin{equation}
\label{eq: 4.18}
\sum_{\textbf{e},\textbf{e}'\in\mathcal{E}_{m}} \frac{\varphi(e)\varphi(e')}{\varphi([\textbf{e},\textbf{e}'])}\sideset{}{'}\sum_{\substack{\textbf{d},\textbf{d}'\in \mathcal{D}_{k}\\d_{i_{j}}=d'_{i_{j}}=1,\forall j=1,...,m+1} }\frac{\varphi(d)\varphi(d')}{\varphi([\textbf{d},\textbf{d}'])}\sum_{\substack{\textbf{d}|\textbf{r}\\ \textbf{e}|\textbf{r}_{0} \\ (r,r_{0})=1}}\frac{y_{\textbf{r},\textbf{r}_{0}}}{\varphi_{\omega}(rr_{0})}\sum_{\substack{\textbf{d}'|\textbf{s}\\ \textbf{e}'|\textbf{s}_{0} \\ (s,s_{0})=1}}\frac{y_{\textbf{s},\textbf{s}_{0}}}{\varphi_{\omega}(ss_{0})}
\end{equation}
$$=\sum_{\substack{\textbf{r},\textbf{s}\in \mathcal{D}_{k}\\r_{i_{j}}=s_{i_{j}}=1,\forall j=1,...,m+1\\ \textbf{r}_{0},\textbf{s}_{0}\in \mathcal{E}_{m}\\ (r,r_{0})=(s,s_{0})=1 }}\frac{y_{\textbf{r},\textbf{r}_{0}}}{\varphi_{\omega}(rr_{0})}\frac{y_{\textbf{s},\textbf{s}_{0}}}{\varphi_{\omega}(ss_{0})}\sideset{}{'}\sum_{\substack{\textbf{d}|\textbf{r},\textbf{d}'|\textbf{s}\\ \textbf{e}|\textbf{r}_{0},\textbf{e}'|\textbf{s}_{0}\\d_{i_{j}}=d'_{i_{j}}=1,\forall j}}\frac{\varphi(d)\varphi(d')\mu(d)\mu(d')}{\varphi([\textbf{d},\textbf{d}'])•}\frac{\varphi(e)\varphi(e')\mu(e)\mu(e')}{\varphi([\textbf{e},\textbf{e}'])•}.$$
Now we restrict $\mathcal{D}_{k}$ asking that $(d_{i}, W_{i}')=1$, for all $i\in\{1,...,k\}$ but $i\neq i_{1},...,i_{m+1}$, where we put 
\begin{equation}
\label{eq: 4.19}
W_{i}'=\prod_{p|W_{i}(h_{i_{1}}-h_{i})\cdots (h_{i_{m+1}}-h_{i})} p.
\end{equation}
In fact, in the sum in the first line of \eqref{eq:4.3} we have no contribution from the $n+h_{i_{j}}$ not primes. Therefore, if $p|d_{i}$, for a certain $i$, than the sum defining $w_{n}$ requires that $p|n+h_{i}$. However, if we also had $p|h_{i_{j}}-h_{i}$, for a certain $j$, then this would imply $p|n+h_{i_{j}}$ and this is not possible because by the support of $\lambda_{\textbf{d}}$ we have $d_{i}<R$. For the sake of simplicity, we will work with the weaker conditions
\begin{equation}
\label{eq: 4.20}
W_{i}'=\prod_{p|W_{i}(h_{i_{1}}-h_{i},...,h_{i_{m+1}}-h_{i})} p.
\end{equation} 
Here $(h_{i_{1}}-h_{i},...,h_{i_{m+1}}-h_{i})$ stands for the greatest common divisor of these $m+1$ differences. We indicate with $\mathcal{D}_{k}'$ the set $\mathcal{D}_{k}$ restricted in this way. 

The inner sum in the second line of \eqref{eq: 4.18} can be written as $\prod_{p|rr_{0}ss_{0}}S_{p}(\textbf{r},\textbf{r}_{0},\textbf{s},\textbf{s}_{0}),$ where 
\begin{equation}
\label{eq: 4.21}
S_{p}(\textbf{r},\textbf{r}_{0},\textbf{s},\textbf{s}_{0})=\sideset{}{'}\sum_{\substack{\textbf{d}|\textbf{r},\textbf{d}'|\textbf{s}\\ \textbf{e}|\textbf{r}_{0},\textbf{e}'|\textbf{s}_{0}\\d_{i_{j}}=d'_{i_{j}}=1,\forall j=1,...,m+1 \\ [e_{1},...,e_{m},e'_{1},...,e'_{m},d_{1},...,d_{k},d'_{1},...,d'_{k}]|p}}\frac{\varphi(de)\varphi(d'e')\mu(de)\mu(d'e')}{\varphi([\textbf{d},\textbf{d}'])\varphi([\textbf{e},\textbf{e}'])}.
\end{equation}
Here $[e_{1},...,e_{m},e'_{1},...,e'_{m},d_{1},...,d_{k},d'_{1},...,d'_{k}]$ indicates the least common divisor of the terms inside it. We easily find that 
\[S_{p}(\textbf{r},\textbf{r}_{0},\textbf{s},\textbf{s}_{0})= \left\{ \begin{array}{lll}
         p-2 & \mbox{if $p|(\textbf{r},\textbf{s})(\textbf{r}_{0},\textbf{s}_{0})$};\\
        -1 & \mbox{if $p\nmid(\textbf{r},\textbf{s})(\textbf{r}_{0},\textbf{s}_{0})$ but $p|rr_{0},ss_{0}$};\\
        0 & \mbox{if $p|rr_{0},p\nmid ss_{0}$ or $p\nmid rr_{0},p|ss_{0}$}.\end{array} \right. \] 
Thus, from the last condition we may assume $rr_{0}=ss_{0}$. 

Using the trivial bound $|y_{\textbf{r},\textbf{r}_{0}}y_{\textbf{s},\textbf{s}_{0}}|\leq \frac{1}{2} (y_{\textbf{r},\textbf{r}_{0}}^{2}+y_{\textbf{s},\textbf{s}_{0}}^{2})$, we see that (by symmetry) the double sum in the last line of \eqref{eq: 4.18} may be bounded by
\begin{equation}
\label{eq: 4.22}
\leq \sum_{\substack{\textbf{r} \in\mathcal{D}_{k}\\ r_{i_{j}}=1,\forall j=1,...,m+1\\ \textbf{r}_{0}\in \mathcal{E}_{m}\\ (r,r_{0})=1 }}\frac{y_{\textbf{r},\textbf{r}_{0}}^{2}}{\varphi_{\omega}(rr_{0})^{2}}\sum_{\substack{\textbf{s} \in\mathcal{D}_{k}\\s_{i_{j}}=1,\forall j=1,...,m+1\\ \textbf{s}_{0}\in \mathcal{E}_{m}\\ ss_{0}=rr_{0}}}\prod_{p|rr_{0}}|S_{p}(\textbf{r},\textbf{r}_{0},\textbf{s},\textbf{s}_{0})|
\end{equation}
$$\leq \sum_{\substack{\textbf{r} \in\mathcal{D}_{k}\\r_{i_{j}}=1,\forall j=1,...,m+1\\ \textbf{r}_{0}\in \mathcal{E}_{m}\\ (r,r_{0})=1 }}y_{\textbf{r},\textbf{r}_{0}}^{2}\prod_{p|rr_{0}}\frac{p-2+\omega(p)-1}{(p-\omega(p))^{2}•}=\sum_{\substack{\textbf{r} \in\mathcal{D}_{k}\\r_{i_{j}}=1,\forall j=1,...,m+1\\ \textbf{r}_{0}\in \mathcal{E}_{m}\\ (r,r_{0})=1 }}y_{\textbf{r},\textbf{r}_{0}}^{2}\prod_{p|rr_{0}}\frac{p+\omega(p)-3}{(p-\omega(p))^{2}•}$$
$$=\sum_{\substack{\textbf{r} \in\mathcal{D}_{k}\\r_{i_{j}}=1,\forall j=1,...,m+1\\ \textbf{r}_{0}\in \mathcal{E}_{m}\\ (r,r_{0})=1 }}\frac{y_{\textbf{r},\textbf{r}_{0}}^{2}}{\prod_{p|rr_{0}}(p+O(k))}.$$
In fact, suppose that $p|rr_{0}$. Then there is just one component among those of $\textbf{s}$ and $\textbf{s}_{0}$ which can be a multiple of $p|(\textbf{r},\textbf{s})(\textbf{r}_{0},\textbf{s}_{0})$ and at most $\omega(p)-1$ possibilities for the components of $\textbf{s},\textbf{s}_{0}$ which can be a multiple of $p\nmid(\textbf{r},\textbf{s})(\textbf{r}_{0},\textbf{s}_{0})$. Indeed, we have exactly $\omega(p)$ indices $i$, with $i\in \{1,...,k\}$, for which $p\nmid W_{i}$ and therefore at most $\omega(p)$ possibilities among the components of $\textbf{s}$ and $\textbf{s}_{0}$ to be multiples of $p$, with just one exception in correspondence of the unique index for which $p|rr_{0}$.

Note that, since $p|rr_{0}$ then $p\nmid WB$ and consequently $p>2k^{2}$. If we take $k$ sufficiently large, we may suppose that $p+O(k)>0$. In order to evaluate the final sum in \eqref{eq: 4.22} we link $y_{\textbf{r},\textbf{r}_{0}}$ with $y^{(m)}_{\textbf{r}}$ and $y_{\textbf{r}_{0}}$ in the following lemma.
\begin{lem}
\label{lem 4.2}
We have 
\begin{equation}
\label{eq: 4.23}
|y_{\textbf{r},\textbf{r}_{0}}|\ll y_{\textbf{r}_{0}}y^{(m)}_{\textbf{r}}\prod_{p|r_{0}}\frac{p}{p-1}\prod_{p|r,p\nmid W_{i_{1}}}\frac{p}{p-1}\cdots \prod_{p|r,p\nmid W_{i_{m}}}\frac{p}{p-1}.
\end{equation}
\end{lem}
\begin{proof}
Inserting the Moebius inversion of \eqref{eq:4.9} and \eqref{eq:4.10} into \eqref{eq:4.8}, we may rewrite $y_{\textbf{r},\textbf{r}_{0}}$ as  
\begin{equation}
\label{eq: 4.24}
y_{\textbf{r},\textbf{r}_{0}}=\mu(rr_{0})\varphi_{\omega}(rr_{0})\sum_{\substack{\textbf{r}|\textbf{d}\\ \textbf{r}_{0}|\textbf{e}\\ (d,e)=1\\ d_{i_{j}}=1,\forall j=1,...,m+1}}\mu(de)\sum_{\textbf{d}|\textbf{f}}\frac{y_{\textbf{f}}^{(m)}}{\varphi_{\omega}(f)}\sum_{\textbf{e}|\textbf{f}_{0}}\frac{y_{\textbf{f}_{0}}}{\varphi(f_{0})}.
\end{equation}
Swapping the order of summation, it becomes
\begin{equation}
\label{eq: 4.25}
\mu(rr_{0})\varphi_{\omega}(rr_{0})\sum_{\substack{\textbf{r}|\textbf{f}\\ \textbf{r}_{0}|\textbf{f}_{0}}}\frac{y_{\textbf{f}}^{(m)}}{\varphi_{\omega}(f)}\frac{y_{\textbf{f}_{0}}}{\varphi(f_{0})}\sum_{\substack{\textbf{r}|\textbf{d}|\textbf{f}\\ \textbf{r}_{0}|\textbf{e}|\textbf{f}_{0} \\(d,e)=1\\ d_{i_{j}}=1,\forall j=1,...,m+1}}\mu(de).
\end{equation}
Using the fact that each $y_{\textbf{f}_{0}}$ is constant and the function $F$ is non-increasing by \eqref{eq:3.8}, we have 
\begin{equation}
\label{eq: 4.26}
|y_{\textbf{r},\textbf{r}_{0}}|\leq \mu^{2}(rr_{0})\varphi_{\omega}(rr_{0})y_{\textbf{r}_{0}}y_{\textbf{r}}^{(m)}\sum_{\substack{\textbf{r}_{0}|\textbf{f}_{0}\in \mathcal{E}_{m}}}\sum_{\substack{\textbf{r}|\textbf{f}\in\mathcal{D}_{k}\\ p|[f,f_{0}], p\nmid (f,f_{0})\Rightarrow p|rr_{0}}}\frac{1}{\varphi_{\omega}(f)\varphi(f_{0})},
\end{equation}
because the inner sum in \eqref{eq: 4.25} is $0$ unless every prime dividing one of $f,f_{0}$ but not the other is a divisor of $rr_{0}$. In this case the sum is $\pm 1$.
We let $f_{i}=r_{i}f_{i}'g_{i}$, with $f_{i}'=\frac{f_{i}}{(f_{i},rr_{0})•}$ and $g_{i}|r_{0}, \forall i=1,...,k$. Moreover, we let $f_{0_{j}}=r_{0_{j}}f_{0_{j}}'g_{0_{j}}$, with $f_{0_{j}}'=\frac{f_{0_{j}}}{(f_{0_{j}},rr_{0})•}$ and $g_{0_{j}}|r,\forall j=1,...,m$. We see the constraint $p|\frac{[f,f_{0}]}{(f,f_{0})}\Rightarrow p|rr_{0}$ means that $f_{0}'=\prod_{j=1}^{m}f_{0_{j}}'=\prod_{i=1}^{k}f_{i}'=f'$. Therefore, we can bound the double sum in \eqref{eq: 4.26} with
\begin{equation}
\label{eq: 4.27}
\frac{1}{\varphi(r_{0})\varphi_{\omega•}(r)}\sum_{\textbf{f}'\in\mathcal{D}_{k}}\frac{\tau_{m}(f')}{\varphi(f')\varphi_{\omega•}(f')}\sum_{\substack{\textbf{g}\in\mathcal{D}_{k} \\ g_{i}|r_{0},\forall i=1,...,k}}\frac{1}{\varphi_{\omega}(g)•}
\sum_{\substack{\textbf{g}_{0}\in\mathcal{E}_{m} \\ g_{0_{j}}|r,\forall j=1,...,m}}\frac{1}{\varphi(g_{0})•}
\end{equation}
$$\leq \frac{1}{\varphi(r_{0})\varphi_{\omega•}(r)}\prod_{p\nmid WB}\bigg(1+\frac{m\omega(p)}{(p-1)(p-\omega(p))}\bigg)\prod_{p|r_{0}}\bigg(1+\frac{\omega(p)}{p-\omega(p)}\bigg) \prod_{p|r,p\nmid W_{i_{1}}}\bigg(1+\frac{1}{p-1}\bigg)\cdots \prod_{p|r,p\nmid W_{i_{m}}}\bigg(1+\frac{1}{p-1}\bigg).$$
The first product is $O(1)$ since it is over primes $p>2k^{2}$ and $k>m$. Thus, we have
\begin{equation}
\label{eq: 4.28}
|y_{\textbf{r},\textbf{r}_{0}}|\ll y_{\textbf{r}_{0}}y_{\textbf{r}}^{(m)}\frac{\mu^{2}(rr_{0})\varphi_{\omega}(rr_{0})}{\varphi(r_{0})\varphi_{\omega}(r)} \prod_{p|r_{0}}\bigg(1+\frac{\omega(p)}{p-\omega(p)}\bigg)\prod_{p|r,p\nmid W_{i_{1}}}\bigg(1+\frac{1}{p-1}\bigg)\cdots \prod_{p|r,p\nmid W_{i_{m}}}\bigg(1+\frac{1}{p-1}\bigg).
\end{equation}
We note that $y_{\textbf{r}_{0}}y_{\textbf{r}}^{(m)}$ is multiplied by
\begin{equation}
\label{eq: 4.29}
\frac{\mu^{2}(rr_{0})\prod_{p|rr_{0}}(p-\omega(p))}{\prod_{p|r_{0}}(p-1)\prod_{p|r}(p-\omega(p))}\prod_{p|r_{0}}\bigg(\frac{p}{p-\omega(p)}\bigg)\prod_{p|r,p\nmid W_{i_{1}}}\bigg(1+\frac{1}{p-1}\bigg)\cdots \prod_{p|r,p\nmid W_{i_{m}}}\bigg(1+\frac{1}{p-1}\bigg)
\end{equation}
$$\leq\prod_{p|r_{0}}\frac{p}{p-1}\prod_{p|r,p\nmid W_{i_{1}}}\frac{p}{p-1}\cdots \prod_{p|r,p\nmid W_{i_{m}}}\frac{p}{p-1},$$
because $(r,r_{0})=1$.
\end{proof}
By Lemma 4.2, the last sum in \eqref{eq: 4.22} may be bounded by 
\begin{equation}
\label{eq: 4.30}
\ll \sum_{\textbf{r}_{0}\in\mathcal{E}_{m}}\frac{y_{\textbf{r}_{0}}^{2}}{\prod_{p|r_{0}}(p+O(k))(1-\frac{1}{p•})^{2}} \sum_{\substack{\textbf{r}\in\mathcal{D}_{k}\\ r_{i_{j}}=1, \forall j=1,...,m+1}}\frac{(y_{\textbf{r}}^{(m)})^{2}}{\prod_{p|r}(p+O(k))h(r)^{2}},
\end{equation}
where we have put
\begin{equation}
\label{eq: 4.31}
h(r)=\prod_{p|r,p\nmid W_{i_{1}}}\bigg(1-\frac{1}{p}\bigg)\cdots \prod_{p|r,p\nmid W_{i_{m}}}\bigg(1-\frac{1}{p}\bigg).
\end{equation}
Now we want to find an estimate on $y_{\textbf{r}}^{(m)}$. We provide it in the next lemma.
\begin{lem} 
\label{lem: 4.3}
We have 
\begin{equation}
\label{eq: 4.32}
y_{\textbf{r}}^{(m)}\ll\left(\frac{\phi(r)}{r}\right)^{m}\left(\frac{WB}{\phi(WB)}\right)^{k-m-1}\mathfrak{S}_{WB}(\mathcal{L})(\log R)^{m}\left((\log R)H'+T_{k}(\log \log R)^{2}H''\right),
\end{equation}
where $H'$ and $H''$ are the integrals in $dt_{i_{1}}\cdots dt_{i_{m+1}}$ of $F_{1}$ and $F_{2}$ respectively, which are evaluated in $(\log r_{i})/(\log R)$ in every positions $i\neq i_{1},...,i_{m+1}$ and $t_{i_{j}}$ elsewhere.
\end{lem}
\begin{proof}
Substituting \eqref{eq:3.7} in \eqref{eq:4.10}, we get
\begin{equation}
\label{eq: 4.33}
y_{\textbf{r}}^{(m)}=\mu(r)\varphi_{\omega}(r)\sum_{\substack{\textbf{r}|\textbf{d}\\ d_{i_{j}}=1, \forall j=1,...,m+1}}\frac{\lambda_{\textbf{d}}}{•\varphi(d)}=\mu(r)\varphi_{\omega}(r)\sum_{\substack{\textbf{r}|\textbf{d}\\ d_{i_{j}}=1, \forall j=1,...,m+1}}\frac{d\mu(d)}{\varphi(d)•}\sum_{\textbf{d}|\textbf{e}}\frac{y_{\textbf{e}}}{\varphi_{\omega}(e)•}
\end{equation}
$$=\mu(r)\varphi_{\omega}(r)\sum_{\textbf{r}|\textbf{e}}\frac{y_{\textbf{e}}}{\varphi_{\omega}(e)•}\sum_{\substack{\textbf{r}|\textbf{d}|\textbf{e}\\ d_{i_{j}}=1, \forall j=1,...,m+1}}\frac{d\mu(d)}{\varphi(d)•}=\frac{\mu^{2}(r)r\varphi_{\omega}(r)}{•\varphi(r)}\sum_{\textbf{r}|\textbf{e}}\frac{y_{\textbf{e}}}{\varphi_{\omega}(e)•}\sum_{\substack{\textbf{f}|\textbf{e}/\textbf{r}\\ f_{i_{j}}=1, \forall j=1,...,m+1}}\frac{f\mu(f)}{\varphi(f)•}$$
$$=\frac{\mu^{2}(r)r\varphi_{\omega}(r)}{•\varphi(r)}\sum_{\textbf{r}|\textbf{e}}\frac{y_{\textbf{e}}}{\varphi_{\omega}(e)•}\prod_{p|e/r}S_{p}'(\textbf{e},\textbf{r}).$$
Here we define $S_{p}'(\textbf{e},\textbf{r})=1$, when $p|e_{j}$ for $j\in\{i_{1},...,i_{m+1}\}$, otherwise $p|e_{j}/r_{j}$ and we put 
\[S_{p}'(\textbf{e},\textbf{r})=\sum_{\substack{\textbf{f}|\textbf{e}/\textbf{r}\\ \textbf{f}\in\mathcal{D}_{k}'\\ f_{i_{j}}=1, \forall j=1,...,m+1\\ [f_{1},...,f_{k}]|p}}\frac{f\mu(f)}{\varphi(f)•}= \left\{ \begin{array}{ll}
         -\frac{1}{p-1} & \mbox{if $p\nmid W_{j}'$};\\
        1 & \mbox{if $p|W_{j}'/W_{j}$}.\end{array} \right. \] 
Now, we let $e_{j}=r_{j}s_{j}t_{j}$ for each $j\neq i_{1},...,i_{m+1}$, where $s_{j}$ is the product of the primes dividing $e_{j}/r_{j}$ but not $W_{j}'$ and $t_{j}$ is the product of primes dividing both $e_{j}/r_{j}$ and $W_{j}'/W_{j}$. We put $s_{i_{j}}=t_{i_{j}}=1$ for every $j\in\{1,...,m+1\}$ and consider the relative $e_{i_{j}}$ separately. For $\textbf{e}\in\mathcal{D}_{k}$ the product $\prod_{p|e/r}S_{p}'(\textbf{e},\textbf{r})$ is then $\mu(s)/\varphi(s)$.
We let $\textbf{r}'$ the vector $\textbf{r}$ in which $r_{i_{j}}$ is replaced with $e_{i_{j}}$, for every component $i_{1},...,i_{m+1}$. By \cite[Lemma 8.2]{M2}, we obtain the following relation
\begin{equation}
\label{eq: 4.34}
y_{\textbf{e}}=y_{\textbf{r}'}+O\bigg(T_{k}Y_{\textbf{r}'}\frac{\log(st)}{\log R•}\bigg),
\end{equation}
where 
$$Y_{\textbf{r}'}=\frac{W^{k}B^{k}\mathfrak{S}_{WB}(\mathcal{L})}{\varphi(WB)^{k}}F_{2}\left(\frac{\log r'_{1}}{\log R•},...,\frac{\log r'_{k}}{\log R•}\right),$$
with $F_{2}$ as in \eqref{eq:3.10}. Inserting this in the last line of \eqref{eq: 4.33}, we obtain 
\begin{equation}
\label{eq: 4.35}
y_{\textbf{r}}^{(m)}=\frac{r}{\varphi(r)•}\sum_{e_{i_{1}},...,e_{i_{m+1}}}\frac{y_{\textbf{r}'}}{\varphi_{\omega}(e)}\sum_{\textbf{s},\textbf{t}}\frac{\mu(s)}{\varphi(s)\varphi_{\omega}(st)•}+O\bigg(\frac{T_{k}}{\log R}\frac{r}{\varphi(r)•}\sum_{e_{i_{1}},...,e_{i_{m+1}}}\frac{Y_{\textbf{r}'}}{\varphi_{\omega}(e)}\sum_{\textbf{s},\textbf{t}}\frac{\log(st)}{\varphi(s)\varphi_{\omega}(st)•} \bigg),
\end{equation}
where now we indicate with $e$ the product $e_{i_{1}}\cdots e_{i_{m+1}}$ and the inner sum is over $\textbf{s}\in\mathcal{D}_{k}', \textbf{t}\in\mathcal{D}_{k}$ subject to $s_{i_{j}}=t_{i_{j}}=1$, for any $j=1,...,m+1$, with $(s,t)=(st, re_{i_{1}}\cdots e_{i_{m+1}})=1$ and $t_{j}|W_{j}'/W_{j}$.

We concentrate first on the main term. We clearly have 
\begin{equation}
\label{eq: 4.36}
\sideset{}{'}\sum_{\textbf{s},\textbf{t}}\frac{\mu(s)}{\varphi(s)\varphi_{\omega}(st)•}=\prod_{p}\sideset{}{'}\sum_{\substack{\textbf{s},\textbf{t}\\ s_{i}|p,t_{i}|p, \forall i}}\frac{\mu(s)}{\varphi(s)\varphi_{\omega}(st)•},
\end{equation}
where $\Sigma'$ means $s_{i_{j}}=t_{i_{j}}=1,$ for any $j\in\{1,...,m+1\}$, $(s,t)=1$, $(s_{i}, W_{i}'re)=(t_{i}, W_{i}re)=1$ and $t_{i}|W_{i}'/W_{i},$ for every $i=1,...,k$. Therefore, we may bound \eqref{eq: 4.36} with
\begin{equation}
\label{eq: 4.37} 
\leq \prod_{p\nmid (W_{i_{1}},...,W_{i_{m+1}})re} \bigg( 1-\frac{\omega(p)-l(p)}{(p-1)(p-\omega(p))•}\bigg)\prod_{\substack{p| (W_{i_{1}},...,W_{i_{m+1}})\\ p\nmid WBr}} \bigg( 1-\frac{\omega(p)-1}{(p-1)(p-\omega(p))•}+\frac{1}{p-\omega(p)•}\bigg),
\end{equation}
where $l(p)=\#\{j\in\{1,...,m+1\}: p\nmid W_{i_{j}}\}$.

In fact, if $p|WBre$ there are no components of $\textbf{s},\textbf{t}$ which can be a multiple of $p$. If $p\nmid WBre$, we have exactly $\omega(p)$ indices $i$ for which $p\nmid W_{i}$. In the case in which $p|(W_{i_{1}},...,W_{i_{m+1}})$, we have exactly $\omega(p)-1$ indices such that $p\nmid W_{i}'$, since we should not consider the chosen index for the residue classes $-h_{i_{1}}\equiv ...\equiv -h_{i_{m+1}}\pmod {p}$. On the other hand, when $p\nmid (W_{i_{1}},...,W_{i_{m+1}})$, among such $\omega(p)$ indices we might count those $i_{j}$ for which $p\nmid W_{i_{j}}$ (and since in this case $s_{i_{j}}=1$), we find at least $\omega(p)-l(p)\geq 0$ components of $\textbf{s}$ that can be a multiple of $p$.\\
If $p\nmid (W_{i_{1}},...,W_{i_{m+1}})re$, then no components of $\textbf{t}$ can be a multiple of $p$, since $p\nmid W_{i}'/W_{i}$ for each $i$.\\ On the other hand, if $p| (W_{i_{1}},...,W_{i_{m+1}})$ and $ p\nmid WBr$, then exactly one component of $\textbf{t}$ can be a multiple of $p$, which is the unique $i$ such that $p|W_{i}'/W_{i}$. Finally, since $(s,t)=1$, no component of $\textbf{s}$ can be a multiple of $p$ if $t$ it is.

We can split \eqref{eq: 4.37} further to:
\begin{equation}
\label{eq: 4.38}
\prod_{p\nmid (W_{i_{1}},...,W_{i_{m+1}})r} \bigg( 1-\frac{\omega(p)-l(p)}{(p-1)(p-\omega(p))•}\bigg)\prod_{p|e} \bigg( 1-\frac{\omega(p)-l(p)}{(p-1)(p-\omega(p)•}\bigg)^{-1}\prod_{\substack{p| (W_{i_{1}},...,W_{i_{m+1}})\\ p\nmid WBr}}\frac{p}{p-1},
\end{equation}
since $(e, r(W_{i_{1}},...,W_{i_{m+1}}))=1$ and 
$$1-\frac{\omega(p)-1}{(p-1)(p-\omega(p))•}+\frac{1}{p-\omega(p)}=\frac{p}{p-1•}.$$
Observe that the first product in \eqref{eq: 4.38} is $\ll 1$, since it is in particular over primes $p\nmid WB$. Inserting \eqref{eq: 4.38} in \eqref{eq: 4.35}, the main term becomes
\begin{equation}
\label{eq: 4.39}
\ll\frac{r}{\phi(r)•}\prod_{\substack{p| (W_{i_{1}},...,W_{i_{m+1}})\\ p\nmid WBr}}\frac{p}{p-1}\sum_{e_{i_{1}},...,e_{i_{m+1}}}\frac{y_{\textbf{r}'}}{g(e)},
\end{equation}
where
$$g(e)=\varphi_{\omega}(e)\prod_{p|e} \bigg( 1-\frac{\omega(p)-l(p)}{(p-1)(p-\omega(p))•}\bigg)$$ 
and it is easily to see that $g(p)=p+O(k)$. Substituting \eqref{eq:3.7} in place of $y_{\textbf{r}'}$, we may write \eqref{eq: 4.39} as
\begin{equation}
\label{eq: 4.40}
\frac{(WB)^{k}\mathfrak{S}_{WB}(\mathcal{L})}{\varphi(WB)^{k}}\frac{r}{\phi(r)•}\prod_{\substack{p| (W_{i_{1}},...,W_{i_{m+1}})\\ p\nmid WBr}}\frac{p}{p-1}\sum_{\substack{e_{i_{1}},...,e_{i_{m+1}}\\ (e_{i_{j}}, rW_{i_{j}})=1,\forall j}}\frac{\tilde{F}}{g(e)},
\end{equation}  
where we indicate with $\tilde{F}$ the function $F$ evaluated in $(\log r_{i})/(\log R)$ in every position $i\neq i_{1},...,i_{m+1}$ and $(\log e_{i_{j}})/(\log R)$ in each $i_{j}$ position. We estimate the sum in \eqref{eq: 4.40} by \cite[Lemma 8.4]{M2}, taking the quantitiy $\Omega_{G}$ in the lemma as $O(kT_{k}^{2})$, and we end up with the following bound
\begin{equation}
\label{eq: 4.41}
\ll \frac{(WB)^{k}\mathfrak{S}_{WB}(\mathcal{L})(\log R)^{m+1}}{\varphi(WB)^{k}}\frac{r}{\phi(r)•}\prod_{\substack{p| (W_{i_{1}},...,W_{i_{m+1}})\\ p\nmid WBr}}\frac{p}{p-1}\prod_{p|r(W_{i_{1}},...,W_{i_{m+1}})}\bigg(1-\frac{1}{p•}\bigg)^{m+1}
\end{equation}
$$\times \prod_{\substack{p:\\ n(p)\geq 1}}\bigg( 1+\frac{n(p)}{g(p)•}\bigg)\bigg(1-\frac{1}{p•}\bigg)^{m+1}\bigg(H+O\bigg(k^{2}T_{k}^{2}\frac{\log\log R}{\log R•}H'\bigg)\bigg),$$
where $H$ and $H'$ are the integrals in $dt_{i_{1}}\cdots dt_{i_{m+1}}$ of $F$ and $F_{1}$ respectively, which are evaluated in $(\log r_{i})/(\log R)$ in every positions $i\neq i_{1},...,i_{m+1}$ and $t_{i_{j}}$ otherwise. Moreover, we have defined 
$$n(p)=\#\lbrace j\in\{1,...,m+1\}: p\nmid rW_{i_{j}}\rbrace.$$ 
We note that 
\begin{equation}
\label{eq: 4.42}
\prod_{\substack{p:\\ n(p)\geq 1}}\bigg( 1+\frac{n(p)}{g(p)•}\bigg)\bigg(1-\frac{1}{p•}\bigg)^{m+1}\ll 1,
\end{equation}
since $n(p)\leq m+1$ and the product is over primes $p>2k^{2}$. Furthermore, we manage the first product in \eqref{eq: 4.41} as
$$\prod_{\substack{p| (W_{i_{1}},...,W_{i_{m+1}})\\ p\nmid WBr}} \frac{p}{p-1}=\frac{\varphi(WB)}{WB•}\prod_{\substack{p| (W_{i_{1}},...,W_{i_{m+1}})\\ p\nmid r}} \frac{p}{p-1},$$
because $(WB, r)=1$ and $WB|(W_{i_{1}},...,W_{i_{m+1}})$, and the second product in \eqref{eq: 4.41} as
$$\prod_{p|r(W_{i_{1}},...,W_{i_{m+1}})}\bigg(1-\frac{1}{p•}\bigg)^{m+1}=\prod_{p|r}\bigg(1-\frac{1}{p•}\bigg)^{m+1}\prod_{\substack{p|(W_{i_{1}},...,W_{i_{m+1}})\\p\nmid r}}\bigg(1-\frac{1}{p•}\bigg)^{m+1}$$
$$=\bigg(\frac{\varphi(r)}{r•}\bigg)^{m+1}\prod_{\substack{p|(W_{i_{1}},...,W_{i_{m+1}})\\p\nmid r}}\bigg(1-\frac{1}{p•}\bigg)^{m+1}.$$
In particular, we observe that 
$$\frac{\varphi(WB)}{WB•}\prod_{\substack{p| (W_{i_{1}},...,W_{i_{m+1}})\\ p\nmid r}}\bigg(1-\frac{1}{p•}\bigg)^{m+1} \frac{p}{p-1}\ll \bigg(\frac{\varphi(WB)}{WB•}\bigg)^{m+1}.$$
Collecting our estimates, we deduce that the main term in \eqref{eq: 4.35} is
\begin{equation}
\label{eq: 4.43}
\ll\frac{r}{\varphi(r)•}\bigg(\frac{\varphi(r)}{r•}\bigg)^{m+1}\bigg(\frac{WB}{\varphi(WB)•}\bigg)^{k}\bigg(\frac{\varphi(WB)}{WB•}\bigg)^{m+1}\mathfrak{S}_{WB}(\mathcal{L})(\log R)^{m+1} \bigg(H+O\bigg(k^{2}T_{k}^{2}\frac{\log\log R}{\log R•}H'\bigg)\bigg)
\end{equation}
$$\ll \bigg(\frac{\varphi(r)}{r•}\bigg)^{m}\bigg(\frac{WB}{\varphi(WB)•}\bigg)^{k-m-1}\mathfrak{S}_{WB}(\mathcal{L})(\log R)^{m+1} H',$$
because clearly $F(t_{1},...,t_{k})\leq F_{1}(t_{1},...,t_{k})$, for every $k$-tuples $(t_{1},...,t_{k})$.

We now return to the error term in \eqref{eq: 4.35}. We use $\log (st)\ll \sqrt{s}(1+\log t)$ and we drop the requirement $(s,t)=1$. In this way, the sum over $s$ factorizes as an Euler product and we get
$$\sum_{\textbf{s}\in\mathcal{D}_{k}'}\frac{\sqrt{s}}{\phi(s)\phi_{\omega}(s)•}\leq \prod_{p>2k^{2}}\left(1+\frac{\sqrt{p}}{(p-1)(p-\omega(p))}\right)\ll 1.$$
We are summing over square-free $t$ with $(t, WBre)=1$ and 
$$t|\Delta=\prod_{\substack{i=1,...,k\\ i\neq i_{1},...,i_{m+1}}}(h_{i_{1}}-h_{i},...,h_{i_{m+1}}-h_{i}).$$
For every such $t$ there is at most one possible $\textbf{t}$. In fact, two cases may happen. If $p|\Delta$ and $p|(W_{i_{1}},...,W_{i_{m+1}})$, there exists a unique index $i\in\{1,...,k\}\setminus \{i_{1},...,i_{m+1}\}$ for which $p|W_{i}'/W_{i}$. It was the chosen index for the residue classes $-h_{i_{1}}\equiv \cdots \equiv -h_{i_{m+1}}\pmod{p}$. When this holds for every $p$ dividing $t$, it gives rise to a unique vector $\textbf{t}$. If otherwise $p|\Delta$ and $p\nmid (W_{i_{1}},...,W_{i_{m+1}})$, there exists an index $j\in\{1,...,m+1\}$ such that $i_{j}$ was the chosen index for the residue class $-h_{i_{j}}\pmod{p}$. In this case there is not any vector $\textbf{t}$.
Thus, the sum over $\textbf{t}$ contributes at most 
$$\sum_{\textbf{t}\in\mathcal{D}_{k}:\ t|\Delta}\frac{1+\sum_{p|t}\log p}{\phi_{\omega}(t)•}\ll\left(1+\sum_{p>2k^{2}:\ p|\Delta}\frac{\log p}{p•}\right)\prod_{p>2k^{2}:\ p|\Delta}\left(1+\frac{1}{\phi_{\omega}(p)}\right)\ll (\log\log \Delta)^{2}\ll (\log\log R)^{2}.$$
Therefore, the error term in \eqref{eq: 4.35} becomes
\begin{equation}
\label{eq: 4.44}
\ll \frac{T_{k}(\log \log R)^{2}}{\log R}\frac{r}{\phi(r)•}\sum_{\substack{e_{i_{1}},...,e_{i_{m+1}}\\ (e_{i_{j}}, WBr)=1, \forall j}}\frac{Y_{\textbf{r}'}}{\phi_{\omega}(e)•},
\end{equation}
relaxing the constraint $(e_{i_{j}}, rW_{i_{j}})=1$ to $(e_{i_{j}}, WBr)=1$.
Substituting the definition of $Y_{\textbf{r}'}$, we should estimate
\begin{equation}
\label{eq: 4.45}
\ll \frac{T_{k}(\log \log R)^{2}}{\log R}\frac{r}{\phi(r)•}\frac{(WB)^{k}\mathfrak{S}_{WB}(\mathcal{L})}{\phi(WB)^{k}•}\sum_{\substack{e_{i_{1}},...,e_{i_{m+1}}\\ (e_{i_{j}}, WBr)=1, \forall j}}\frac{\tilde{F}_{2}}{\phi_{\omega}(e)•},
\end{equation} 
where we have indicated with $\tilde{F}_{2}$ the function $F_{2}$ evaluated in $(\log r_{i})/(\log R)$ in every position $i\neq i_{1},...,i_{m+1}$ and $(\log e_{i_{j}})/(\log R)$ in each $i_{j}$ position. Now, since we can write 
$$F_{2}(t_{1},...,t_{k})=\sum_{j=1}^{k}G_{j}(t_{j})\prod_{\substack{1\leq i\leq k\\ i\neq j}}G_{i}(t_{i}),$$ 
for every $k$-tuples $(t_{1},...,t_{k})$, for certain functions $G_{1},...,G_{k}$, we apply a simple variation of the Lemma 8.4 in \cite{M2}, in which we allow for different functions $G_{i}$ instead of a single one, but verifying the same conditions present in the Lemma. In this way, we can bound \eqref{eq: 4.45} with 
\begin{equation}
\label{eq: 4.46}
\ll \frac{T_{k}(\log \log R)^{2}}{\log R}\frac{r}{\phi(r)•}\frac{(WB)^{k}\mathfrak{S}_{WB}(\mathcal{L})}{\phi(WB)^{k}•}\left(\frac{\phi(WBr)}{WBr}\right)^{m+1}\prod_{p\nmid WBr}\left(1+\frac{m+1}{p-\omega(p)}\right)\left(1-\frac{1}{p•}\right)^{m+1}(\log R)^{m+1}H'',
\end{equation}
where $H''$ is the integral in $dt_{i_{1}}\cdots dt_{i_{m+1}}$ of $F_{2}$, which is evaluated in $(\log r_{i})/(\log R)$ in every positions $i\neq i_{1},...,i_{m+1}$ and $t_{i_{j}}$ otherwise.
Clearly, \eqref{eq: 4.46} is equal to 
\begin{equation}
\label{eq: 4.47}
T_{k}(\log \log R)^{2}(\log R)^{m}\left(\frac{\phi(r)}{r}\right)^{m}\left(\frac{WB}{\phi(WB)}\right)^{k-m-1}\mathfrak{S}_{WB}(\mathcal{L})H'',
\end{equation}
because $(WB,r)=1$ and the product is $\ll 1$, since it is over primes $p>2k^{2}$. This concludes the estimate of the error term in \eqref{eq: 4.35} and the proof of the Lemma \ref{lem: 4.3}.
\end{proof}
Now we return to (4.30) and we firstly estimate the second sum which, after inserting (4.32), becomes:
\begin{equation}
\label{eq: 4.48}
\ll \left(\frac{WB}{\varphi(WB)}\right)^{2(k-m-1)}\mathfrak{S}_{WB}^{2}(\mathcal{L})(\log R)^{2(m+1)}\sum_{\substack{\textbf{r}\in\mathcal{D}_{k}\\ r_{i_{j}}=1, \forall j=1,...,m+1}}\frac{(H')^{2}}{\prod_{p|r}(p+O(k))}
\end{equation}
$$+\left(\frac{WB}{\varphi(WB)}\right)^{2(k-m-1)}\mathfrak{S}_{WB}^{2}(\mathcal{L})(\log R)^{2m}T^{2}_{k}(\log\log R)^{4}\sum_{\substack{\textbf{r}\in\mathcal{D}_{k}\\ r_{i_{j}}=1, \forall j=1,...,m+1}}\frac{(H^{''})^{2}}{\prod_{p|r}(p+O(k))},$$
because 
$$\frac{1}{h(r)^{2}•}\bigg(\frac{\varphi(r)}{r•}\bigg)^{2m}=\bigg(\frac{\varphi(r)}{r•}\bigg)^{2m}\prod_{\substack{p|r\\p\nmid W_{i_{1}}}}\bigg(\frac{p}{p-1•}\bigg)^{2}\cdots\prod_{\substack{p|r\\p\nmid W_{i_{m}}}}\bigg(\frac{p}{p-1•}\bigg)^{2}\leq 1.$$
We start working with the first sum in (4.48). By \cite[Lemma 8.4]{M2} we obtain the following bound
\begin{equation}
\label{eq: 4.49}
\ll L_{k}(F_{1})(\log R)^{k-m-1}\prod_{p\nmid WB}\bigg(1+\frac{\omega(p)-m-1}{p+O(k)•}\bigg)\bigg(1-\frac{1}{p•}\bigg)^{k-m-1}\prod_{p|WB}\bigg(1-\frac{1}{p•}\bigg)^{k-m-1},
\end{equation}
and we note that the two products over the primes are respectively $\ll\mathfrak{S}_{WB}^{-1}(\mathcal{L})$ and $\left(\frac{\varphi(WB)}{WB•}\right)^{k-m-1}$. Here, we define
$$L_{k}(F_{1})=\idotsint\limits_{\{t_{i}\geq 0:\ i\not\in \{i_{1},...,i_{m+1}\}\}}\biggl(\idotsint\limits_{t_{i_{1}},...,t_{i_{m+1}}\geq 0}F_{1}(t_{1},...,t_{k})dt_{i_{1}}\cdots dt_{i_{m+1}}\biggr)^{2}dt_{1}\hat{\cdots} dt_{k},$$
where $dt_{1}\hat{\cdots} dt_{k}$ means that we are differentiating only by the $dt_{i}$'s for $i\neq i_{1},...,i_{m+1}$. By the symmetry of $F_{1}$ with respect to each variable, we may rewrite $L_{k}(F_{1})$ in the following simpler form
$$L_{k}(F_{1})=\int_{0}^{\infty}\cdots \int_{0}^{\infty}\biggl(\int_{0}^{\infty}\cdots \int_{0}^{\infty}F_{1}(t_{1},...,t_{k})dt_{1}\cdots dt_{m+1}\biggr)^{2}dt_{m+2}\cdots dt_{k}.$$
Similarly, but this time using the little variation of \cite[Lemma 8.4]{M2} mentioned above, the second sum in \eqref{eq: 4.48} may be bounded by
\begin{equation}
\label{eq: 4.50}
\ll L_{k}(F_{2})(\log R)^{k-m-1}\mathfrak{S}_{WB}^{-1}(\mathcal{L})\left(\frac{\varphi(WB)}{WB•}\right)^{k-m-1},
\end{equation}
where as before we let 
$$L_{k}(F_{2})=\int_{0}^{\infty}\cdots \int_{0}^{\infty}\biggl(\int_{0}^{\infty}\cdots \int_{0}^{\infty}F_{2}(t_{1},...,t_{k})dt_{1}\cdots dt_{m+1}\biggr)^{2}dt_{m+2}\cdots dt_{k}.$$
Therefore, the second sum in \eqref{eq: 4.30} is
\begin{equation}
\label{eq: 4.51}
\ll (\log R)^{k+m+1}\mathfrak{S}_{WB}(\mathcal{L})\left(\frac{WB}{\phi(WB)}\right)^{k-m-1}\left( L_{k}(F_{1})+\frac{T_{k}^{2}(\log \log R)^{4}}{(\log R)^{2}•}L_{k}(F_{2})\right).
\end{equation}
Arguing in the same way as in \cite[Lemma 8.6]{M2}, we find the following estimates
$$L_{k}(F_{1})\ll \frac{(\log k)^{m+1}}{k^{m+1}}I_{k}(F)$$
and 
$$L_{k}(F_{2})\ll k^{2}L_{k}(F_{1})\ll \frac{(\log k)^{m+1}}{k^{m-1}}I_{k}(F).$$
In conclusion, we obtain
\begin{equation}
\label{eq: 4.52}
\sum_{\substack{\textbf{r}\in\mathcal{D}_{k}\\ r_{i_{j}}=1, \forall j=1,...,m+1}}\frac{(y_{\textbf{r}}^{(m)})^{2}}{\prod_{p|r}(p+O(k))h(r)^{2}}\ll \frac{(\log k)^{m+1}}{k^{m+1}}I_{k}(F)(\log R)^{k+m+1}\mathfrak{S}_{WB}(\mathcal{L})\left(\frac{WB}{\varphi(WB)}\right)^{k-m-1}.
\end{equation} 
Before going on, let us explicit the constant $y_{\textbf{r}_{0}}$. To this aim we compute the sum $\sum_{\textbf{r}_{0}\in\mathcal{E}_{m}}\mu^{2}(r_0)/\phi(r_0)$.

By a trivial application of \cite[Lemma 8.4]{M2}, we get
\begin{equation}
\label{eq: 4.53}
\sum_{\textbf{r}_{0}\in\mathcal{E}_{m}}\frac{\mu^{2}(r_0)}{\phi(r_0)}\geq\sum_{(r_{0_{j}}, W_{i_{j}})=1, \forall j}\frac{\mu^{2}(r_{0_{1}}\cdots r_{0_{m}})}{\phi(r_{0_{1}}\cdots r_{0_{m}})}\prod_{j=1}^{m}\textbf{1}_{[0,1]}\left(\frac{\log r_{0_{j}}}{\log R^{1/3m}}\right)
\end{equation}
$$\gg \left(\frac{\log R}{3m}\right)^{m}\prod_{p|(W_{i_{1}},...,W_{i_{m}})}\left(1-\frac{1}{p•}\right)^{m}\prod_{p:\ n(p)\geq 1}\left(1+\frac{n(p)}{p-1}\right)\left(1-\frac{1}{p•}\right)^{m},$$
where $n(p)=\#\{j\in\{1,...,m\}: p\nmid W_{i_{j}}\}$. We indicate the last product in \eqref{eq: 4.53} as $P(W_{i_{1}},...,W_{i_{m}})$ and we note that it converges, since for all the primes $p\nmid \prod_{i\neq j}(h_{j}-h_{i})$ we have $p\nmid W_{i}$, for every $i=1,...,k$. Anyway, it seems difficult to prove that $P(W_{i_{1}},...,W_{i_{m}})\gg 1$, even if we believe this is the case.

Now we can choose $y_{\textbf{r}_{0}}$ as 
$$y_{\textbf{r}_{0}}=(3m)^{m}\left(\frac{(W_{i_{1}},...,W_{i_{m}})}{\phi((W_{i_{1}},...,W_{i_{m}}))}\right)^{m}P(W_{i_{1}},...,W_{i_{m}})^{-1},$$
so that we immediately find $\tilde{\lambda}_{(1,...,1)}\gg (\log R)^{m}$. We call the value of $y_{\textbf{r}_{0}}$ as $\bar{y}$. We note that for small values of $m$ the expression of $\bar{y}$ is easy and computable. For example, one can prove that
$$\bar{y}\ll \left\{ \begin{array}{ll}
        \frac{W_{i_{1}}}{\phi(W_{i_{1}})} & \mbox{if $m=1$};\\
        \frac{(W_{i_{1}},W_{i_{2}})}{\phi((W_{i_{1}},W_{i_{2}}))}\frac{[W_{i_{1}},W_{i_{2}}]}{\phi([W_{i_{1}},W_{i_{2}}])} &     \mbox{if $m=2$}. \end{array} \right. $$
Regarding the first sum in \eqref{eq: 4.30}, by $m$ applications of \cite[Lemma 8.4]{M2}, we find:
\begin{equation}
\label{eq: 4.54}
\sum_{\textbf{r}_{0}\in\mathcal{E}_{m}}\frac{y_{\textbf{r}_{0}}^{2}\mu^{2}(r_{0})}{\prod_{p|r_{0}}(p+O(k))(1-\frac{1}{p•})^{2}}\leq \bar{y}^{2}\prod_{j=1}^{m}\sum_{(r_{0_{j}}, W_{i_{j}})=1}\frac{\mu^{2}(r_{0_{j}})}{\prod_{p|r_{0_{j}}}(p+O(k))}\textbf{1}_{[0,1]}\left(\frac{\log r_{0_{j}}}{\log R}\right)
\end{equation}
$$\ll\bar{y}^{2}(\log R)^{m}\frac{\phi(W_{i_{1}})}{W_{i_{1}}}\cdots\frac{\phi(W_{i_{m}})}{W_{i_{m}}}\prod_{p\nmid W_{i_{1}}}\bigg(1+\frac{1}{p+O(k)}\bigg)\bigg(1-\frac{1}{p•}\bigg)\cdots\prod_{p\nmid W_{i_{m}}}\bigg(1+\frac{1}{p+O(k)}\bigg)\bigg(1-\frac{1}{p•}\bigg).$$
We note that each product over primes, in the last line of \eqref{eq: 4.54}, is bounded. Moreover, we can use the following estimate
$$\frac{\phi(W_{i_{1}})}{W_{i_{1}}}\cdots\frac{\phi(W_{i_{m}})}{W_{i_{m}}}\leq \left(\frac{\varphi(WB)}{WB}\right)^{m}.$$
Thus, \eqref{eq: 4.54} reduces to 
\begin{equation}
\label{eq: 4.55}
\leq C_{1}^{m}\bar{y}^{2}(\log R)^{m}\left(\frac{\varphi(WB)}{WB}\right)^{m},
\end{equation}
for a suitable constant $C_{1}>0$. Using the estimates \eqref{eq: 4.52} and \eqref{eq: 4.55} we find
\begin{equation}
\label{eq: 4.56}
\sum_{\substack{\textbf{r} \in\mathcal{D}_{k}\\r_{i_{j}}=1,\forall j\\ \textbf{r}_{0}\in \mathcal{E}_{m}\\ (r,r_{0})=1 }}\frac{y_{\textbf{r},\textbf{r}_{0}}^{2}}{\prod_{p|rr_{0}}(p+O(k))}\leq C_{2}^{m}\bar{y}^{2}\left(\frac{\log k}{k}\right)^{m+1}I_{k}(F)(\log R)^{k+2m+1}\mathfrak{S}_{WB}(\mathcal{L})\left(\frac{WB}{\phi(WB)}\right)^{k-2m-1},
\end{equation}
for a certain $C_{2}>0$. Collecting the results, we get that \eqref{eq: 4.5} can be estimated by
\begin{equation}
\label{eq: 4.57}
\leq C_{3}^{m} \bar{y}^{2}\left(\frac{\log k}{k•}\right)^{m+1}\frac{I_{k}(F)}{\varphi(W)•}x(\log x)^{k}\mathfrak{S}_{WB}(\mathcal{L})\left(\frac{WB}{\varphi(WB)}\right)^{k-2m-1},
\end{equation}
for a suitable constant $C_{3}>0$. This is also the final bound of \eqref{eq:4.3}. In fact, by Lemma \ref{lem 4.1} and by using \cite[Lemma 8.1 (i)]{M2} and \cite[Lemma 8.6]{M2} to take into  consideration the size of $\mathfrak{S}_{WB}(\mathcal{L})$ and of $I_{k}(F)$, we easily see that the error term coming from \eqref{eq:4.7} is negligible compared to \eqref{eq: 4.57}.

Finally, we recall that we have to sum the bound \eqref{eq: 4.57} over all the residue classes $v_{0}\pmod{W}$, which is equivalent to multiply it by
$$\varphi_{\omega}(W)=W\prod_{p|W}\bigg(1-\frac{\omega(p)}{•p}\bigg)\bigg(1-\frac{1}{p•}\bigg)^{-k}\left(\frac{\varphi(W)}{W}\right)^{k}=W\mathfrak{S}_{WB}(\mathcal{L})^{-1}\mathfrak{S}_{B}(\mathcal{L})\left(\frac{\varphi(W)}{W}\right)^{k}.$$
In this way, we can estimate our main sum in \eqref{eq: 1.4} with
\begin{equation}
\label{eq: 4.58}
\leq C_{4}^{m}\bar{y}^{2}\left(\frac{\log k}{k•}\right)^{m+1}I_{k}(F)x(\log x)^{k}\bigg(\frac{\varphi(W)•}{W}\bigg)^{k-1}\left(\frac{WB}{\varphi(WB)}\right)^{k-2m-1}\mathfrak{S}_{B}(\mathcal{L}),
\end{equation}
for a certain $C_{4}>0$. When $h_{1},...,h_{k}\leq k^{2}$ (in fact, it suffices that for every prime $p\nmid WB$ happens that $p\nmid (h_{j}-h_{i})$, for all $i\neq j$) then all the factors $W_{i}=WB$ and consequently $\bar{y}$ reduces to 
$$\bar{y}\ll(3m)^{m}\left(\frac{WB}{\phi(WB)}\right)^{m}.$$
This leads to a simplification of the expression \eqref{eq: 4.58} of the form
\begin{equation}
\leq C^{m}m^{m}\left(\frac{\log k}{k•}\right)^{m+1}I_{k}(F)x(\log x)^{k}\left(\frac{B}{\varphi(B)}\right)^{k-1}\mathfrak{S}_{B}(\mathcal{L}),
\end{equation}
for a suitable $C>0$, since $(W,B)=1$. The proofs of Theorem 1.1 and Corollary 1.2 are completed.
\\

\end{document}